\documentclass{amsart}

\usepackage{amssymb,mathtools}
\usepackage[shortlabels]{enumitem}
\usepackage[alphabetic,initials]{amsrefs}

\usepackage{hyperref}

\newcommand{\br}{\overline}

\newcommand{\C}{\mathbb C}

\newcommand{\N}{\mathbb N}

\theoremstyle{plain}
\newtheorem{theorem}{Theorem}
\newtheorem{lemma}[theorem]{Lemma}
\newtheorem{conjecture}{Conjecture}
\newtheorem{prop}[theorem]{Proposition}

\theoremstyle{definition}
\newtheorem{definition}[theorem]{Definition}

\theoremstyle{remark}

\DeclareMathOperator{\dist}{\textup{\text{dist}}}
\DeclareMathOperator{\diam}{\textup{\text{diam}}}

\DeclareMathOperator{\inter}{\textup{\text{int}}}

\numberwithin{equation}{section}
\numberwithin{theorem}{section}
\numberwithin{conjecture}{section}

\begin{document}
\title{Non-removability of Sierpi\'nski carpets}
\author{Dimitrios Ntalampekos}
\address{Institute for Mathematical Sciences, Stony Brook University, Stony Brook, NY 11794, USA.}
\email{dimitrios.ntalampekos@stonybrook.edu}

\date{\today}
\thanks{The author was partially supported by NSF Grant DMS-1506099}

\begin{abstract}
We prove that all Sierpi\'nski carpets in the plane are non-removable for (quasi)conformal maps. More precisely, we show that for any two Sierpi\'nski carpets $S,S'\subset \widehat\C$ there exists a homeomorphism $f\colon \widehat \C\to \widehat\C$ that is conformal in $\widehat\C\setminus S$ and it maps $S$ onto $S'$. As a corollary, we obtain a partial answer to a question of Bishop \cite{Bishop:flexiblecurves}, whether any planar compact set with empty interior and positive measure can be mapped to a set of measure zero with an exceptional homeomorphism of the plane, conformal off that set.
\end{abstract}
\maketitle

\section{Introduction}

In recent work \cite[Theorem 1.8]{Ntalampekos:gasket} the current author proved that the Sierpi\'nski gasket, also called the Sierpi\'nski triangle, is non-removable for conformal maps. In the same paper the author provided some evidence that \textit{all} homeomorphic copies of the Sierpi\'nski gasket should be  non-removable for conformal maps; see \cite[Theorem 1.7]{Ntalampekos:gasket}. This gave birth to the conception that ``some sets should be non-removable for topological reasons". Sierpi\'nksi carpets are topologically ``larger" than gaskets and provide a perfect candidate to test this heuristic. In this work we prove that this is actually the case:

\begin{theorem}\label{theorem:nonremovable}
All Sierpi\'nski carpets $S\subset \widehat \C$ are non-removable for conformal maps.
\end{theorem}

This result resolves a conjecture of the author \cite[Conjecture 1]{Ntalampekos:gasket}. Here, we pose another more general conjecture:
\begin{conjecture}Every planar compact set containing a homeomorphic copy of $C\times [0,1]$ is non-removable for conformal maps, where $C$ is the middle-thirds Cantor set.
\end{conjecture}

We provide some necessary definitions before stating our other results. 

\begin{definition}\label{definition:removable}
Let $K\subset \widehat \C$ be a compact set. We say that $K$ is \textit{conformally removable} or \textit{removable for conformal maps} if any homeomorphism $f\colon \widehat \C \to \widehat \C$ that is conformal in $\widehat \C\setminus K$ is actually conformal in $\widehat \C$ and thus, it is M\"obius.
\end{definition}
An equivalent notion is that of quasiconformal removability. We direct the reader to \cite{Younsi:removablesurvey}, \cite{Ntalampekos:gasket}, and the references therein for more background.

The \textit{standard Sierpi\'nski carpet} is constructed by subdividing the unit square $[0,1]^2$ into nine squares of sidelength $1/3$ and removing the middle square, and then proceeding inductively in each of the remaining eight squares. It can be easily proved that the standard Sierpi\'nski carpet is non-removable for conformal maps; see e.g.\ the discussion in \cite[Section 1.1]{Ntalampekos:gasket}. 

In general, a \textit{Sierpi\'nski carpet} $S\subset \widehat \C$ is a set homeomorphic to the standard carpet. It is a fundamental result of Whyburn \cite{Whyburn:theorem} that a set $S$ is a Sierpi\'nski carpet if and only if $S$ has empty interior and $S=\widehat{\C} \setminus \bigcup_{i\in \N} Q_i$, where $\{Q_i\}_{i\in \N}$ is a family of Jordan regions with disjoint closures and (spherical) diameters converging to $0$. The regions $Q_i$, $i\in \N$, are called the \textit{peripheral disks} and the boundaries $\partial Q_i$, $i\in \N$, are called the \textit{peripheral circles} of $S$. Our main result is the following.

\begin{theorem}\label{theorem:homeo}
Let $S,S'\subset \widehat\C$ be Sierpi\'nski carpets. Then there exists a homeomorphism $f\colon \widehat\C\to \widehat\C$ with $f(S)=S'$ such that $f$ is conformal on $\widehat\C\setminus S$. Moreover, the image of any finite collection of peripheral circles can be prescribed: if $C_1,\dots,C_N$ and $C_1',\dots,C_N'$ are peripheral circles of $S$ and $S'$, respectively, then $f$ can be chosen such that $f(C_i)=C_i'$ for all $i=1,\dots,N$.
\end{theorem}

One can easily construct carpets $S'$ with positive Lebesgue measure. The theorem, combined with this fact, proves immediately Theorem \ref{theorem:nonremovable}, since a map $f$ that sends a carpet $S$ of measure zero or positive measure to a carpet $S'$ of positive measure or measure zero, respectively, cannot be M\"obius. Our proof of Theorem \ref{theorem:homeo} is topological and utilizes the ideas of Whyburn \cite{Whyburn:theorem}.

Using Theorem \ref{theorem:homeo} we obtain a partial answer to a question raised by Bishop \cite[Question 3]{Bishop:flexiblecurves}. He asked whether any compact set $K\subset \C$ with empty interior and positive area can be mapped to a set of measure zero with a homeomorphism $f\colon \C\to \C$ that is conformal on $\C\setminus K$. A partial answer to that question was given by Kaufman and Wu \cite[Theorem 3]{KaufmanWu:removable}, where the authors proved that there exists a \textit{subset} of $K$ with positive, strictly smaller measure such that the answer to the question is positive for that subset. We prove that one can \textit{``enlarge"} $K$ to a set $L$ by attaching to it a small set of measure zero and Hausdorff dimension arbitrarily close to $1$, so that the answer is positive for $L$:

\begin{prop}\label{prop:zeromeasure}
Let $K\subset \C$ be a compact set with empty interior. Then for each $\varepsilon>0$ there exists a compact set $L\supset K$ such that the Hausdorff dimension of $L\setminus K$ is at most $1+\varepsilon$ and the following holds:
there exists a homeomorphism $f\colon \C\to \C$ that is conformal on $\C\setminus L$ such that $f(L)$ has Hausdorff dimension at most $1+\varepsilon$ and in particular Lebesgue measure zero.
\end{prop}
In fact $L$ can be taken to lie in the $\varepsilon$-neighborhood of $K$ if we allow $f$ to be quasiconformal, but we will not go into details for the sake of brevity.

\subsection*{Acknowledgments}
I would like to thank my PhD advisor Mario Bonk at UCLA for many useful conversations on the problem of non-removability of carpets and for his notes on the proof of Whyburn, which were very enlightening. I also thank Stanislav Smirnov for reigniting my interest on the problem by asking me the question whether carpets are removable and the anonymous referee for reading the manuscript carefully and providing useful comments.

\section{Proof of Theorem \ref{theorem:homeo}}

In the entire section we only use the spherical metric of $\widehat{\C}$. We start with the following definition adapted from \cite{Whyburn:theorem}. 

\begin{definition}\label{def:subdivision}
Let $\varepsilon>0$ and $Y$ be a closed subset of $\widehat{\C}$. A partition $\{Y_{\alpha}\}$ of $Y$ is called an \textit{$\varepsilon$-subdivision} if the collection $\{Y_{\alpha}\}$ consists of finitely many closed Jordan regions, also called $2$-cells, with disjoint interiors and diameter less than $\varepsilon$.

A partition $\{S_{\alpha}\}$ of a Sierpi\'nski carpet $S\subset \widehat{\C}$ into Sierpi\'nski carpets is an \textit{$\varepsilon$-subdivision (rel.\ $Q_1, \dots, Q_N$)} if there exist peripheral disks $Q_1,\dots,Q_N$ of $S$ and an $\varepsilon$-subdivision $\{Y_{\alpha}\}$ of the closed domain $Y\coloneqq \widehat{\C}\setminus \bigcup_{i=1}^N Q_i$, for which 
\begin{enumerate}[(1)]
\item the boundaries of the $2$-cells $Y_{\alpha}$ are contained in $S\setminus \bigcup_{i=N+1}^\infty \partial Q_i$, where $Q_i$, $i\geq N+1$, are the remaining peripheral disks of $S$, and
\item $\{S_{\alpha}\} =\{S\cap Y_{\alpha}\}$, i.e., the Sierpi\'nski carpets in the subdivision $\{S_{\alpha}\}$ arise as the intersections of the $2$-cells in the subdivision $\{Y_{\alpha}\}$ with the original Sierpi\'nski carpet $S$.
\end{enumerate}
\end{definition}

We remark that each peripheral disk of $S$ is either one of the disks $Q_1,\dots,Q_N$, or it is a peripheral disk of one of the carpets $S_{{\alpha}}$ in the $\varepsilon$-subdivision of $S$.

\begin{lemma}\label{lemma:subdivisions}
Let $S,S'\subset \widehat\C$ be Sierpi\'nski carpets. Fix peripheral disks $Q_1,\dots,Q_N$ and $Q_1',\dots,Q_N'$ of $S$ and $S'$, respectively.  Then for each $\varepsilon>0$ there exists $M>N$ and peripheral disks $Q_{N+1},\dots,Q_M$ and $Q_{N+1}',\dots,Q_{M}'$ of $S$ and $S'$, respectively, such that there exists an $\varepsilon$-subdivision of $S$ rel.\ $Q_1,\dots,Q_M$ and an $\varepsilon$-subdivision of $S'$ rel.\ $Q_1',\dots,Q_M'$. Moreover, for each family of orientation-preserving homeomorphisms $g_i\colon \br{Q_i}\to  \br{Q_i'}$, $i\in \{1,\dots,M\}$, the $\varepsilon$-subdivisions can be taken in such a way that they correspond to each other under a homeomorphism $g\colon \widehat\C\to \widehat\C$ that extends $g_i$, $i\in \{1,\dots,M\}$.  
\end{lemma}

Here the correspondence of the subdivisions under $g$ is be understood in the sense that each 2-cell $Y_{\alpha}$ of the subdivision of $S$ is mapped by $g$ onto a 2-cell $Y'_{\alpha}$ of the subdivision of $S'$.

A weaker version of this lemma is used by Whyburn in his topological characterization of the carpet; see \cite[Lemma 1]{Whyburn:theorem}. For the sake of completeness and to make this work self-contained, we include an outline of the proof of Lemma \ref{lemma:subdivisions} later in Section \ref{section:subdivisions}.

\begin{proof}[Proof of Theorem \ref{theorem:homeo}]
Let $\varepsilon_k=1/k$. We fix peripheral circles $C_1,\dots,C_N$ and $C_1',\dots,C_N'$ of $S$ and $S'$, respectively, as in the statement of the theorem, and we consider conformal maps $g_i$, $i\in \{1,\dots,N\}$, from the peripheral disk $Q_i$ of $S$ bounded by $C_i$ onto the peripheral disk $Q_i'$ of $S'$ bounded by $C_i'$. These conformal maps extend to the boundaries by Carath\'eodory's theorem and provide orientation-preserving homeomorphisms $g_i\colon \br {Q_i}\to \br {Q_i'}$. By Lemma \ref{lemma:subdivisions}, we can find peripheral disks $Q_{N+1},\dots,Q_M$ and $Q_{N+1}',\dots,Q_M'$ of $S$ and $S'$, respectively, with conformal maps $g_i\colon \br {Q_i}\to \br {Q_i'}$, $i\in \{N+1,\dots,M\}$, and we can find $\varepsilon_1$-subdivisions of $S$ and $S'$ that correspond to each other under a global homeomorphic extension $f_1\colon \widehat\C\to \widehat\C$  of the maps $g_i\colon \br {Q_i}\to \br {Q_i'}$. Observe that $\dist(f_1(S),S')\leq \varepsilon_1$ and $\dist(S,f_1^{-1}(S'))\leq \varepsilon_1$.

Next, we fix one of the Sierpi\'nski carpets $S_{{\alpha}}$ in the $\varepsilon_1$-subdivision of $S$ that corresponds to a carpet $S_{{\alpha}}'$ in the $\varepsilon_1$-subdivision of $S'$; that is, the corresponding 2-cells are mapped to each other under $f_1$ (however, we do not necessarily have $f_1(S_{{\alpha}})=S_{{\alpha}}'$). Note that the peripheral circles of $S_{{\alpha}}$ and $S_{{\alpha}}'$ necessarily have diameters bounded above by $\varepsilon_1$. 

There is a distinguished peripheral circle of $S_{{\alpha}}$ ($S_{{\alpha}}'$) that separates it from the other carpets in the subdivision of $S$ ($S'$). We call $R_1$ ($R_1'$) the peripheral disk of $S_{{\alpha}}$ ($S_{{\alpha}}'$) bounded by that peripheral circle. Consider the orientation-preserving homeomorphism $h_1\coloneqq f_1\big|_{\br {R_1}}\colon \br {R_1}\to \br{R_1'}$. Now we apply Lemma \ref{lemma:subdivisions} on $S_{{\alpha}}$ and $S_{{\alpha}}'$, to obtain peripheral disks $R_2,\dots,R_{K}$ and $R_2',\dots,R_{K}'$ of $S_{{\alpha}}$ and $S'_{{\alpha}}$, respectively, together with conformal maps $h_i\colon \br{R_i} \to \br{R_i'}$, $i\in \{2,\dots,K\}$, and we can find $\varepsilon_2$-subdivisions of $S_{{\alpha}}$ and $S_{{\alpha}}'$ (rel.\ $R_i$ and $R_i'$, $i=1,\dots,K$, respectively) that correspond to each other under a homeomorphic extension $f_{{\alpha},2}\colon \widehat{\C}\setminus R_1\to \widehat{\C}\setminus R_1'$ of the maps $h_i$, $i\in \{2,\dots,K\}$.

We repeat this procedure for each of the carpets $S_{\alpha}$ in the $\varepsilon_1$-subdivision of $S$. If we collect the $\varepsilon_2$-subdivisions of the carpets of the form $S_{{\alpha}}\subset S$ and $S_{{\alpha}}'\subset S'$, we obtain $\varepsilon_2$-subdivisions of the carpets $S$ and $S'$. Patching together the resulting maps of the form $f_{{\alpha},2}$ yields a homeomorphism $f_2\colon \widehat\C\to \widehat\C$ under which the $\varepsilon_1$- and $\varepsilon_2$-subdivisions of $S,S'$ are in correspondence. The map $f_2$ agrees with $f_1$ on $\bigcup_{i=1}^M \br {Q_i}$ and maps conformally the peripheral disks of $S$ having diameter larger than $\varepsilon_2$ onto some peripheral disks of $S'$. Note that by construction $f_2^{-1}$ maps conformally the peripheral disks of $S'$ having diameter larger than $\varepsilon_2$ onto some peripheral disks of $S$. Moreover, the $L^\infty$-distance of $f_1$ and $f_2$ is bounded by $\varepsilon_1$; the same statement holds for the inverses of the maps. Also, note that $\dist(f_2(S),S')\leq \varepsilon_2$ and $\dist(S,f_2^{-1}(S'))\leq \varepsilon_2$.

We repeat the procedure of the last two paragraphs for each carpet in the $\varepsilon_2$-subdivision of $S,S'$. Inductively, for each $k\in \N$ we obtain a homeomorphism $f_k$ of $\widehat\C$ with the following properties:
\begin{enumerate}
\item\label{prop:convergence} The $L^\infty$-distance of $f_k,f_m$ and the $L^\infty$-distance of $f_k^{-1},f_m^{-1}$ are bounded by $\varepsilon_k$ for $m\geq k$.
\item\label{prop:distance}$\dist(f_k(S),S')\leq \varepsilon_k$ and $\dist(S,f_k^{-1}(S')) \leq \varepsilon_k$ for all $k\in \N$.
\item\label{prop:conformal} For each $k\in \N$ the map $f_k$ maps conformally peripheral disks $Q_{i_1},\dots,Q_{i_k}$ of $S$ onto corresponding peripheral disks $Q_{i_1}',\dots, Q_{i_k}'$ of $S'$. In fact, for each $l\in \N$ the sequence $\{f_k\}_{k\in \N}$ is eventually constant on $Q_{i_l}$.
\item\label{prop:exhaust} The sequences of peripheral disks $Q_{i_k},Q_{i_k}'$, $k\in \N$, exhaust the sequences of peripheral disks of $S,S'$, respectively.
\end{enumerate}

Since $\varepsilon_k\to 0$, \eqref{prop:convergence} implies that $f_k$ converges uniformly in $\widehat{\C}$ to a homeomorphism $f$ of $\widehat{\C}$. By \eqref{prop:distance}, we have $f(S)=S'$. Finally, \eqref{prop:conformal} and \eqref{prop:exhaust} imply that each peripheral disk of $S$ is mapped conformally onto a peripheral disk of $S'$. In particular, $f$ is conformal on $\widehat{\C}\setminus S$.
\end{proof}

\section{Proof of Proposition \ref{prop:zeromeasure}}

An easily verifiable fact that we will use here is that for each $\varepsilon>0$ there exists a self-similar square carpet $S\subset \C$ with Hausdorff dimension bounded above by $1+\varepsilon$. 

We fix $\varepsilon>0$. Suppose that $K\subset \C$ is a compact set with empty interior and let $B$ be a large ball containing $K$. We consider a dyadic square decomposition of the open set $B\setminus K$. In each dyadic square $Q\subset B\setminus K$ in this decomposition we consider a square carpet $S_Q$ with Hausdorff dimension bounded by $1+\varepsilon$, scaled so that it ``fits" in the square $S$; that is to say that the boundary of the unbounded peripheral disk of $S_Q$ is precisely the boundary of $Q$. Then we take $L$ to be the closure of the union of all carpets $S_Q$. 

Note that $L\supset K$, since $\inter(K)=\emptyset$, and that $L$ is a Sierpi\'nski carpet, as one can see by Whyburn's criterion (see the Introduction): $\inter(L)=\emptyset$ (since $\inter(K)=\emptyset$) and the complement of $L$ in $\widehat \C$ consists of countably many Jordan regions with disjoint closures and diameters shrinking to $0$. Moreover, $L\setminus K$ has Hausdorff dimension at most $1+\varepsilon$, since it is the union of the circle $\partial B$ with  countably many carpets $S_Q$, each having Hausdorff dimension at most $1+\varepsilon$.

Now, we apply Theorem \ref{theorem:homeo}, which provides us with a homeomorphism $f\colon \widehat \C \to \widehat\C$ that maps the carpet $L$ onto a fixed square carpet with Hausdorff dimension at most $1+\varepsilon$. In particular, $f(L)$ has Lebesgue measure zero. By post-composing with a M\"obius transformation of $\widehat{\C}$, which does not affect the Hausdorff dimension, we may assume that $f(\infty)=\infty$ and thus, we have a homeomorphism $f\colon \C\to  \C$ that is conformal in $\C\setminus L$, as desired.\qed

\section{Outline of proof of Lemma \ref{lemma:subdivisions}}\label{section:subdivisions}

The proof relies on the following version of Moore's theorem \cite{Moore:theorem}:
\begin{theorem}{\cite[Corollary 6A, p.\ 56]{Daverman:decompositions}}\label{moore}\index{Moore's theorem}
Let $\{R_i\}_{i\in \N}$ be a sequence of Jordan regions in the sphere $\widehat\C$ with mutually disjoint closures and diameters converging to $0$, and consider an open set $\Omega \supset \bigcup_{i\in \N} \br {R_i}$. Then there exists a continuous, surjective map $F\colon \widehat{\C} \to S^2$ that is the identity outside $\Omega$ and it induces the decomposition of $\widehat \C$ into the sets $\{\br {R_i}\}_{i\in \N}$ and points. In other words, there are countably many distinct points $p_i\in S^2$, $i\in \N$, such that $F^{-1}(p_i)= \br {R_i}$ for $i\in \N$ and $F$ is injective on $\widehat\C\setminus \bigcup_{i\in \N}\br {R_i}$ with  $F( \widehat\C\setminus \bigcup_{i\in \N} \br {R_i})= S^2\setminus \{p_i:i\in \N\}$.
\end{theorem}

Here $S^2$ is identical to $\widehat{\C}$, but it will be more convenient to have different notation for the target and view $S^2$ mostly from a topological point of view.

First, we consider the peripheral disks $Q_1,\dots,Q_N$ of $S$ and $Q_1',\dots,Q_N'$ of $S'$ given in the statement of Lemma \ref{lemma:subdivisions} and fix $\varepsilon>0$. We append to these collections some peripheral disks $Q_{N+1},\dots,Q_M$ of $S$ and $Q_{N+1}',\dots,Q_M'$ of $S'$ so that the remaining peripheral disks of $S$ and $S'$ have diameters less than $\varepsilon$. Now, we apply Theorem \ref{moore} to the region $\Omega=\widehat{\C}\setminus \bigcup_{i=1}^{M} \br{Q_i}$ and to the remaining peripheral disks $\{R_i\}_{i\in \N}$ of $S$ contained in $\Omega$, all of which have diameter less than $\varepsilon$. This yields a ``collapsing" map $F\colon \widehat{\C}\to S^2$ that collapses each $R_i$ to a point $p_i$, but is injective otherwise. We apply analogously Theorem \ref{moore} to the carpet $S'$ and obtain a map $F'$ that collapses peripheral disks $R_i'$ to points $p_i'$.

Now, we consider given orientation-preserving homeomorphisms $g_i\colon \br {Q_i}\to \br {Q_i'}$, $i\in \{1,\dots,M\}$, as in the statement of Lemma \ref{lemma:subdivisions}. These homeomorphisms ``project" down to homeomorphisms 
$$\widetilde g_i\coloneqq  F'\circ g_i \circ \bigl(F\big|_{\br {Q_i}}\bigr)^{-1}  \colon  F(\br{Q_i})\to F'(\br{Q_i'}), \quad 1\leq i\leq M.$$
These are orientation-preserving homeomorphisms between finitely many disjoint Jordan regions in $S^2$. This implies that there exists a homeomorphic extension $\widetilde g \colon S^2\to S^2$ of $\widetilde g_i$, $i\in \{1,\dots,M\}$.

Consider the countable set $A'=\{\widetilde g(p_i): i\in \N\} \cup \{p_i': i\in \N\}\subset S^2$. This is the union of the set of points $p_i'$ arising from collapsing the peripheral disks of $S'$ (under $F'$) that are different from $Q_1',\dots,Q_M'$, together with the images $\widetilde{g}(p_i)$ of the corresponding points $p_i$ arising from collapsing the peripheral disks of $S$ (under $F$). We fix a $\delta>0$ (to be specified) and a $\delta$-subdivision $\{Y_{\alpha}'\}$ of the closed domain $Y'=S^2\setminus \bigcup_{i=1}^M F'(Q_i')$ so that the boundaries $\partial Y_{\alpha}'$ avoid the countable set $A'$.

We fix a Jordan region $Y_\alpha'$ and consider the preimage $(F')^{-1}(\partial Y_\alpha')$, which is a Jordan curve since $F'$ is injective on  it. Let $X_\alpha'\subset \widehat{\C}$ be the Jordan region bounded by that Jordan curve with the property that $(F')^{-1}(p_i') \subset X_\alpha'$ for some $p_i'\in Y_\alpha'$; in fact, this will hold for all $p_i'\in Y_\alpha'$, as one can see using homotopy arguments. Then the Jordan regions $\{X_\alpha'\}$ give a subdivision of the closed domain $\widehat{\C}\setminus \bigcup_{i=1}^M \br{Q_i'}$ in the sense of Definition \ref{def:subdivision} and the boundaries of the $2$-cells $X_\alpha'$ are contained in $S'\setminus \bigcup_{i=M+1}^\infty \partial Q_i'$ (here $Q_i'$, $i\geq M+1$, are the remaining peripheral disks of $S'$). Finally, the sets $S_\alpha'\coloneqq S'\cap Y_\alpha'$ are Sierpi\'nski carpets by Whyburn's criterion (stated before Theorem \ref{theorem:homeo}), since $S_\alpha'\subset S'$ has empty interior and the complement of $S_\alpha'$ consists of Jordan regions with disjoint closures, whose diameters converge to $0$; what is important here is that $\partial X_\alpha'= (F')^{-1}(\partial Y_\alpha')$ does not intersect any peripheral circle $\partial Q_i'$, $i\geq M+1$. Therefore, the collection $\{S_\alpha'\}$ is a subdivision of the carpet $S'$, as in Definition \ref{def:subdivision}. %Taking preimages of the boundaries $\partial Y_{\alpha}'$ under $F'$ yields a subdivision (as in Definition \ref{def:subdivision}) of the carpet $S'$; this can be justified using homotopy arguments. Note that $F'$ is injective on each boundary $(F')^{-1}(\partial Y_{\alpha}')$, by Theorem \ref{moore}. We denote by $ X_{\alpha}'\subset \widehat{\C}$ the closed Jordan region bounded by $(F')^{-1}(\partial Y_{\alpha}')$ that contains a carpet $S_{{\alpha}}'$ in the subdivision of $S'$. 
One can ensure that this subdivision is an $\varepsilon$-subdivision by choosing a sufficiently small $\delta>0$ in the previous paragraph and using the following elementary lemma:
\begin{lemma}
Suppose that $X$ and $Y$ are compact metric spaces and $h\colon X\to Y$ is a continuous, surjective map with the property that the preimage of each point has diameter at most a given number $\varepsilon>0$. Then there exists a $\delta>0$ such that for each set $B\subset Y$ with $\diam(B)<\delta$ we have $\diam(h^{-1}(B))<\varepsilon$. 
\end{lemma}

Now, we consider the preimages $Y_{\alpha}$ under $\widetilde g$ of the Jordan regions $Y_{\alpha}'$ lying in the partition of $Y'=S^2\setminus \bigcup_{i=1}^M F'(Q_i')$. The sets $Y_{\alpha}$ provide a partition of $Y= S^2\setminus \bigcup_{i=1}^M F(Q_i)$ into Jordan regions and the boundaries $\partial Y_{\alpha}$ of the Jordan regions in the partition avoid the countable set $A=\{p_i:i\in \N\}$. One now takes preimages under $F$ of the Jordan curves $\partial Y_{\alpha}$; these provide a subdivision of $S$ as in the previous paragraph. We denote by $X_{\alpha}\subset \widehat{\C}$ the Jordan region bounded by $F^{-1}(\partial Y_{\alpha})$ that contains a carpet $S_{{\alpha}}$ in the subdivision of $S$. Again, if $\delta$ is chosen to be sufficiently small, then by continuity the regions $Y_{\alpha}$ will be small enough, so that the preceding lemma can guarantee that the sets $X_{\alpha}$ yield an $\varepsilon$-subdivision of $S$. 

The map $\widetilde g$ (composed appropriately with $F$ and $(F')^{-1}$) provides a homeomorphism $g$ from the union of the Jordan curves of the form $\partial X_{\alpha}$ onto the union of the Jordan curves of the form $\partial X_{\alpha}'$. This homeomorphism extends the maps $g_i \colon \partial { Q_i}\to \partial{ Q_i'}$, $i\in \{1,\dots,M\}$. The union of the Jordan curves $\partial X_{\alpha}$ is the $1$-skeleton of a $2$-cell decomposition of $\widehat{\C}$, in which the $2$-cells are the sets $X_{\alpha}$, together with the sets $\br {Q_i}$, $i\in \{1,\dots,M\}$. The analogous statement holds for $X_{\alpha}'$ and $\br{Q_i'}$, $i\in \{1,\dots,M\}$. Now, one can extend the homeomorphism $g$ to a homeomorphism of $\widehat{\C}$; see e.g.\ \cite[Chapter 5.2]{BonkMeyer:Thurston}. By construction, the $\varepsilon$-subdivisions of $S$ and $S'$ are in correspondence under $g$, in the sense of the comments following the statement of Lemma \ref{lemma:subdivisions}.\qed

\bibliography{biblio}

% \bib, bibdiv, biblist are defined by the amsrefs package.
\begin{bibdiv}
\begin{biblist}

\bib{Bishop:flexiblecurves}{article}{
      author={Bishop, Christopher~J.},
       title={Some homeomorphisms of the sphere conformal off a curve},
        date={1994},
     journal={Ann. Acad. Sci. Fenn. Ser. A I Math.},
      volume={19},
      number={2},
       pages={323\ndash 338},
}

\bib{BonkMeyer:Thurston}{book}{
      author={Bonk, Mario},
      author={Meyer, Daniel},
       title={Expanding {T}hurston maps},
      series={Mathematical Surveys and Monographs},
   publisher={American Mathematical Society, Providence, RI},
        date={2017},
      volume={225},
}

\bib{Daverman:decompositions}{book}{
      author={Daverman, Robert~J.},
       title={Decompositions of manifolds},
      series={Pure and Applied Mathematics},
   publisher={Academic Press, Inc., Orlando, FL},
        date={1986},
      volume={124},
}

\bib{KaufmanWu:removable}{article}{
      author={Kaufman, Robert},
      author={Wu, Jang-Mei},
       title={On removable sets for quasiconformal mappings},
        date={1996},
     journal={Ark. Mat.},
      volume={34},
      number={1},
       pages={141\ndash 158},
}

\bib{Moore:theorem}{article}{
      author={Moore, Robert~L.},
       title={Concerning upper semi-continuous collections of continua},
        date={1925},
     journal={Trans. Amer. Math. Soc.},
      volume={27},
      number={4},
       pages={416\ndash 428},
}

\bib{Ntalampekos:gasket}{article}{
      author={Ntalampekos, Dimitrios},
       title={{Non-removability of the {S}ierpi{\'n}ski gasket}},
        date={2019},
     journal={Invent. Math.},
      volume={216},
      number={2},
       pages={519\ndash 595},
}

\bib{Whyburn:theorem}{article}{
      author={Whyburn, Gordon~Thomas},
       title={Topological characterization of the {S}ierpi{\'n}ski curve},
        date={1958},
     journal={Fund. Math.},
      volume={45},
       pages={320\ndash 324},
}

\bib{Younsi:removablesurvey}{article}{
      author={Younsi, Malik},
       title={On removable sets for holomorphic functions},
        date={2015},
     journal={EMS Surv. Math. Sci.},
      volume={2},
      number={2},
       pages={219\ndash 254},
}

\end{biblist}
\end{bibdiv}

\end{document}